\documentclass[11pt, a4paper,reqno]{amsart}

\usepackage[british]{babel}
\usepackage{hyperref}
\usepackage{tikz}
\usetikzlibrary{arrows.meta,calc,shapes.geometric}

\calclayout
\hypersetup{colorlinks=true,linkcolor=blue,citecolor=blue,urlcolor=blue}

\newtheorem{theorem}{Theorem}[section]
\newtheorem{proposition}[theorem]{Proposition}
\newtheorem{lemma}[theorem]{Lemma}
\theoremstyle{definition}
\newtheorem{remark}[theorem]{Remark}
\numberwithin{equation}{section}

\newcommand{\R}{\mathbb{R}}
\newcommand{\Hn}{\mathcal H}
\newcommand{\dd}{\,\mathrm{d}}
\newcommand{\A}{\mathring A}
\renewcommand{\L}{\operatorname{L}} 
\newcommand{\C}{\operatorname{C}} 

\title[Spherical rigidity for an exterior overdetermined problem]{Spherical rigidity for an exterior overdetermined problem with Neumann data prescribed by mean curvature}
\author{Lukas Niebel}
\address[Lukas Niebel]{Institut f\"ur Analysis und Numerik, Universit\"at M\"unster\\
Orl\'eans-Ring 10, 48149 M\"unster, Germany.}
\email{lukas.niebel@uni-muenster.de}
\date{8 April 2026}
\thanks{Lukas Niebel is funded by the Deutsche Forschungsgemeinschaft (DFG, German Research Foundation) under Germany's Excellence Strategy EXC 2044/2--390685587, Mathematics M\"unster: Dynamics--Geometry--Structure.}



\begin{document}
\begin{abstract}
    We study an overdetermined elliptic free boundary problem for exterior domains in $\mathbb{R}^N$, $N \ge 2$, introduced by F.~Morabito [\emph{Comm. PDE} \textbf{46} (2021), 1137--1161]. The overdetermining condition prescribes the Neumann data as a multiple of the boundary mean curvature, with parameter $\Gamma$, together with a spherical compatibility condition. For $N \ge 3$, we prove rigidity of the spherical solution among star-shaped domains when $\Gamma \ge N-2$; in the borderline case $\Gamma = N-2$, the star-shapedness assumption can be removed, and rigidity holds among all bounded domains. The proof combines the Pohozaev identity, geometric identities, and the sharp boundary inequality of Agostiniani and Mazzieri for capacitary potentials. We also obtain rigidity among bounded domains for $\Gamma \le 0$ via Serrin's moving plane method. In dimension two, the unit disc is the only admissible domain for every $\Gamma$.
\end{abstract}

\maketitle

\section{Introduction}

Let $N\ge 3$, $u_0 \in \R $, $\gamma \in \R$, and let $R_0>0$. In this paper we study bounded domains $\Omega\subset\R^N$ with boundary $\Sigma:=\partial\Omega$ such that
\begin{equation}\label{eq:original-problem}
    \left\{
    \begin{aligned}
        -\Delta u      & = 0              &  & \text{in } \R^N\setminus\overline\Omega, \\
        u              & = u_0            &  & \text{on } \Sigma,                       \\
        \partial_\nu u & = \gamma \Hn + C &  & \text{on } \Sigma,                       \\
        u(x)           & \to 0            &  & \text{as } |x|\to\infty,
    \end{aligned}
    \right.
\end{equation}
with $|\Omega|=\omega_N R_0^N$. Here $\nu$ is the outer unit normal to $\Omega$, $\omega_N=|B_1|$, and $\Hn$ denotes the normalised mean curvature, with the convention
\[
    \Hn_{\partial B_R}=-\frac1R.
\]
Thus spheres have negative normalised mean curvature with respect to the outer normal.

We first treat the case $N \ge 3$ and discuss the planar case $N = 2$ in Section~\ref{sec:planar}. If $\Omega=B_{R_0}$, then the exterior harmonic function with boundary value $u_0$ is unique and radial, hence the overdetermined condition can be satisfied only for one specific value of $C$. This is the spherical compatibility value that we must fix before discussing rigidity or bifurcation around the ball. The problem \eqref{eq:original-problem} with $\Omega=B_{R_0}$ admits a solution if and only if
\begin{equation}\label{eq:C0}
    C=C_0(\gamma,R_0):=\frac{\gamma-(N-2)u_0}{R_0}.
\end{equation}
In that case the solution is
\[
    u_*(x)=u_0\Big(\frac{R_0}{|x|}\Big)^{N-2},
    \qquad |x|\ge R_0.
\]

In fact,
\[
    \partial_\nu u_*\big|_{\partial B_{R_0}}=-\frac{(N-2)u_0}{R_0},
    \qquad
    \Hn_{\partial B_{R_0}}=-\frac1{R_0}.
\]
Therefore the boundary condition $\partial_\nu u=\gamma \Hn + C$ becomes
\[
    -\frac{(N-2)u_0}{R_0}=-\frac{\gamma}{R_0}+C,
\]
which is \eqref{eq:C0}. Accordingly, from now on we fix the value of $C$ to be $C_0(\gamma,R_0)$.

\subsection{Rescaling to unit volume}

Let us first treat the case $u_0 = 0$. Let $\Omega$ be of class $\C^2$. By the maximum principle $u = 0$ in $\R^{N} \setminus \overline{\Omega}$, hence $\partial_\nu u = 0$ and thus $\Hn$ is constant whenever $\gamma \neq 0$. By Alexandrov's soap bubble theorem we deduce that $\Omega$ is, up to translation, a ball. If $\gamma = 0$, then any domain $\Omega$ yields a solution to \eqref{eq:original-problem}.

Let us assume from now on $u_0 \neq 0$. Set
\[
    x=R_0 y,
    \qquad
    u(x)=u_0 v(y),
    \qquad
    \widetilde\Omega=R_0^{-1}\Omega,
    \qquad
    \Gamma:=\frac\gamma{u_0}.
\]
Since $|\Omega|=\omega_N R_0^N$, we have $|\widetilde\Omega|=\omega_N$. The curvature scales by
\[
    \Hn_{\partial\Omega}(R_0 y)=\frac1{R_0}\,\Hn_{\partial\widetilde\Omega}(y),
\]
and the rescaled boundary condition becomes
\[
    \frac{u_0}{R_0}\partial_\nu v
    =
    \frac{\gamma}{R_0}\Hn_{\partial\widetilde\Omega}
    +
    \frac{\gamma-(N-2)u_0}{R_0}.
\]
Dividing by $u_0/R_0$ yields
\[
    \partial_\nu v=\Gamma \Hn_{\partial\widetilde\Omega}+\Gamma-(N-2).
\]
Dropping tildes and renaming $v$ as $u$, we are reduced to the normalised problem
\begin{equation}\label{eq:normalized-problem}
    \left\{
    \begin{aligned}
        -\Delta u      & =0                          &  & \text{in }\R^N\setminus\overline\Omega, \\
        u              & =1                          &  & \text{on }\Sigma,                       \\
        \partial_\nu u & = \Gamma \Hn + \Gamma-(N-2) &  & \text{on }\Sigma,                       \\
        u(x)           & \to 0                       &  & \text{as }|x|\to\infty,
    \end{aligned}
    \right.
\end{equation}
with $|\Omega|=\omega_N$.

Thus, after fixing the spherical compatibility value of $C$, the rigidity problem depends only on the single dimensionless parameter $\Gamma$, and the reference spherical solution is the unit ball.

\begin{remark} \label{rem:regularity}
    Any $\C^{1,1}$ solution $(u,\Omega)$ of \eqref{eq:normalized-problem} with $\Gamma\neq 0$
    is immediately smooth. Indeed, since $\Sigma=\partial\Omega$ is $\C^{1,1}$, classical
    boundary regularity for harmonic functions on $\C^{1,\alpha}$ hypersurfaces gives
    \[
        \partial_\nu u\in \C^\alpha(\Sigma)\qquad\text{for every }\alpha\in(0,1),
    \]
    see \cite{MR125307,MR162050}. Hence
    \[
        \Hn=\frac{\partial_\nu u-\Gamma+(N-2)}{\Gamma}\in \C^\alpha(\Sigma).
    \]
    Writing $\Sigma$ locally as a graph, the function $g$ solves a uniformly elliptic prescribed-mean-curvature equation with $\C^\alpha$ right-hand side
    and therefore $g\in \C^{2,\alpha}$ by the regularity theory for the prescribed
    mean-curvature equation \cite{lian2024pointwiseregularitylocallyuniformly}. Once $\Sigma\in \C^{2,\alpha}$, the classical
    boundary Schauder estimates for the harmonic Dirichlet problem \cite{MR125307,MR162050},
    combined with the corresponding Schauder bootstrap for the mean-curvature equation
    \cite{lian2024pointwiseregularitylocallyuniformly}, yield $\Sigma,u\in \C^\infty$.
\end{remark}

\subsection{Literature}

This problem was introduced in \cite{MR4267505} for $N = 3$, where a bifurcation analysis similar to that in \cite{MR2058166} was used to construct symmetry-breaking solutions. The same problem in dimensions $N \ge 2$ has been investigated in \cite{MR4917176}.

For $\Gamma = 0$ this problem is the exterior Serrin problem (\cite{MR333220,MR333221}) with constant Neumann data studied by Reichel in \cite{MR1463801}. Exceptional domains for a related problem are investigated in \cite{MR4603682,MR4985459}. We refer to \cite{MR3192039} for the classification of solutions in the planar case.
The reconstruction of domains from non-constant Neumann boundary data is studied, for example, in \cite{MR2746441}.

Similar overdetermined free boundary problems are studied in the context of electrostatics in \cite{MR2058166} ($N = 3$).
Moreover, the problem of \cite{MR2058166} is also related to fluid dynamics, where it describes stationary hollow vortex sheets with surface tension ($N = 2$). Symmetry-breaking solutions were constructed in \cite{MR1794849} for $N = 2$. The author proved in \cite{niebel2025globalrigiditytwodimensionalbubbles} that the circle is globally rigid in the regime where the surface tension, i.e.\ the curvature term, dominates ($N = 2$).
We want to mention also \cite{MR4972964} where steady bubbles in inviscid fluids are described by a related problem.
The difference to the problem of the present article is that here the jump condition is linear in the Neumann trace, whereas in \cite{MR2058166,MR1794849,niebel2025globalrigiditytwodimensionalbubbles,MR4972964} it depends quadratically on the trace.
Another overdetermined free boundary value problem involving the mean curvature and the Neumann trace is studied in \cite{MR2834910,MR2921691}.

\subsection{Bifurcation analysis}

Morabito proved in dimension $N=3$ that nontrivial exterior domains bifurcate from the complement of a ball for a sequence of distinguished values of the parameter on the spherical compatibility line \eqref{eq:C0}; see \cite{MR4267505}. In the higher-dimensional extension of Dai, Liu, and Morabito, the same bifurcation mechanism is carried out for $N\ge 4$ in a class of perturbations of the ball; see \cite{MR4917176}. Translating their parameter $\gamma$ into our normalised parameter $\Gamma=\gamma/u_0$, the bifurcation values become
\begin{equation}\label{eq:Gamma3}
    \Gamma_\ell^{(3)}=\frac{2}{\ell+2},
    \qquad \ell\ge 2,
\end{equation}
and, for $N\ge 4$,
\begin{equation}\label{eq:GammaN}
    \Gamma_\ell^{(N)}=\frac{(N-1)(N-2)}{\ell+N-1},\qquad \ell\ge 2.
\end{equation}
This can be verified by a simple calculation of the linearisation of $\partial_\nu u = \Gamma \Hn + \Gamma-(N-2)$ for graphs over the sphere.
The mode $\ell=1$ corresponds to translations and is excluded in the bifurcation analysis. Moreover, if we do not fix the volume we have another bifurcation point at $\Gamma_{0}^{(N)} = N-2$ which corresponds to the $\ell = 0$ mode and a radial branch.

Although \cite{MR4267505,MR4917176} are not formulated on the fixed-volume slice, the same distinguished values also arise in the linearisation of the volume-constrained problem. Indeed, near the unit ball the condition $|\Omega|=\omega_N$ defines a smooth codimension-one Banach manifold of perturbations, whose tangent space at the ball consists of functions with zero average on $\mathbb S^{N-1}$. Since the bifurcating directions are spherical harmonics of degree $\ell\ge 2$, they have zero average and are therefore tangent to the fixed-volume manifold. Thus the fixed-volume constraint is compatible with the bifurcation picture at the linearised level. However, this first-order compatibility does not by itself imply that the Crandall--Rabinowitz argument of \cite{MR4267505,MR4917176} can be carried out directly on the fixed-volume slice. To obtain genuine local fixed-volume branches one would need an additional constrained Lyapunov--Schmidt reduction, or an equivalent two-parameter argument. We do not investigate this further.

We emphasise that in the correct functionalanalytic setup this analysis implies local rigidity modulo translations of close-to-spherical solutions whenever $\Gamma \notin \{ \Gamma_\ell^{(N)} : \ell \ge 2 \}$.

\subsection{Global rigidity of the unit ball}

At the value $\Gamma=N-2$, rigidity is a direct corollary of Agostiniani--Mazzieri's sharp boundary inequality for capacitary potentials \cite[Theorem 4.1]{MR4037467}.

\begin{theorem}\label{thm:critical-rigidity}
    Let $N\ge 3$ and $\Gamma=N-2$. Let $\Omega\subset\R^N$ be a bounded domain with $|\Omega|=\omega_N$ and $\partial\Omega\in \C^{1,1}$. Assume that \eqref{eq:normalized-problem} admits a solution $u \in \C^2(\R^N\setminus\overline\Omega)\cap \C^1(\R^N\setminus\Omega)$.
    Then $\Omega$ is a unit ball. Consequently
    \[
        u(x)=|x-x_0|^{2-N}
        \qquad\text{for some }x_0\in\R^N,
    \]
    and $\Sigma=\partial B_1(x_0)$.
\end{theorem}

Assuming additional star-shapedness, we obtain rigidity for all $\Gamma >N-2$.

\begin{theorem}\label{thm:star-outer-rigidity}
    Let $N\ge 3$ and $\Gamma>N-2$. Let $\Omega\subset\R^N$ be a bounded domain with $|\Omega|=\omega_N$ and $\partial\Omega\in \C^{1,1}$.
    Assume that $\Omega$ is star-shaped with respect to some point $x_0$.
    Assume that \eqref{eq:normalized-problem} admits a solution $u \in \C^2(\R^N\setminus\overline\Omega)\cap \C^1(\R^N\setminus\Omega)$. Then $\Omega$ is a unit ball.
\end{theorem}

Next, we record the classical case $\Gamma=0$, due to Reichel
\cite[Theorem~1]{MR1463801}.

\begin{theorem}[Reichel]\label{thm:gamma-zero}
    Let $N\ge 3$ and $\Gamma=0$. Let $\Omega\subset\R^N$ be a bounded domain
    with $|\Omega|=\omega_N$ and $\partial\Omega\in \C^2$.
    Assume that \eqref{eq:normalized-problem} admits a solution
    $u \in \C^2(\R^N\setminus\overline\Omega)\cap \C^1(\R^N\setminus\Omega)$.
    Then $\Omega$ is a unit ball.
\end{theorem}

In the regime $\Gamma < 0$ the sign of the curvature term is favourable for Serrin's method of moving planes and thus we can extend Reichel's proof to solutions of \eqref{eq:normalized-problem}.

\begin{theorem}\label{thm:negative-rigidity}
    Let $N\ge 3$ and $\Gamma<0$. Let $\Omega\subset\R^N$ be a bounded domain with $|\Omega|=\omega_N$ and $\partial\Omega\in \C^{1,1}$. Assume that \eqref{eq:normalized-problem} admits a solution $u \in \C^2(\R^N\setminus\overline\Omega)\cap \C^1(\R^N\setminus\Omega)$.
    Then $\Omega$ is a unit ball.
\end{theorem}

\begin{figure}[ht]
    \centering
    \begin{tikzpicture}[
            x=1.008cm,y=1cm,>=Latex,
            every node/.style={font=\scriptsize},
            global/.style={black,line width=1.9pt},
            local/.style ={black,line width=1.7pt,dash pattern=on 4.5pt off 2pt},
            outer/.style ={black,line width=1.6pt,dash pattern=on 0pt off 2.1pt,line cap=round}
        ]
        \def\rad{2.25pt}

        \coordinate (A)  at (-1.00,0);
        \coordinate (O)  at (0,0);

        \coordinate (E1) at (0.22,0);
        \coordinate (E2) at (0.40,0);
        \coordinate (E3) at (0.62,0);
        \coordinate (E4) at (0.92,0);
        \coordinate (E5) at (1.55,0);
        \coordinate (E6) at (2.45,0);
        \coordinate (E7) at (3.75,0);

        \coordinate (C)  at (5.75,0);
        \coordinate (B)  at (7.10,0);

        \draw[global] (A) -- (O);
        \draw[local]  (O) -- (C);
        \draw[outer,-{Latex[length=2.2mm]}] (C) -- (B);

        \foreach \P in {E1,E2,E3,E4,E5,E6,E7}
            {
                \fill[white] (\P) circle (\rad);
                \draw[black,line width=0.75pt] (\P) circle (\rad);
            }

        \fill[black] (O) circle (\rad);
        \fill[black] (C) circle (\rad);
        \draw[black,line width=0.75pt] (O) circle (\rad);
        \draw[black,line width=0.75pt] (C) circle (\rad);

        \node[below=5pt] at (O) {$0$};
        \node[below=5pt] at (E5) {$\Gamma_4$};
        \node[below=5pt] at (E6) {$\Gamma_3$};
        \node[below=5pt] at (E7) {$\Gamma_2$};
        \node[below=5pt] at (C) {$N-2$};

    \end{tikzpicture}
    \caption{Schematic rigidity/bifurcation picture (not to scale).
        The solid black part and the filled black points indicate global rigidity, the open circles indicate bifurcation values,
        the dashed part denotes the regime of close-to-spherical local rigidity, and the dotted ray
        denotes global rigidity in the class of star-shaped domains.}
\end{figure}
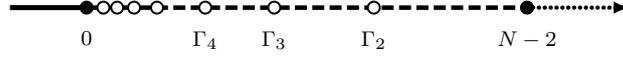

\begin{remark}
    We emphasise that the bifurcation branches of \cite{MR4267505,MR4917176} are small $\C^2$ perturbations of the sphere and thus star-shaped. Comparing the bifurcation regime with the rigidity threshold we see that
    \[
        0<\Gamma_\ell^{(N)}=\frac{(N-1)(N-2)}{\ell+N-1}<N-2
        \qquad (\ell\ge2),
    \]
    with gap of size $(N-2)-\Gamma_2^{(N)}=2-\frac{6}{N+1}$. We emphasise that $\Gamma_{0}^{(N)} = N-2$ and in this sense our result is sharp.
\end{remark}

\begin{remark}
    Every bounded convex domain is star-shaped with respect to every interior point. Hence Theorem~\ref{thm:star-outer-rigidity} yields rigidity for convex sets in the regime $\Gamma>N-2$.
\end{remark}

\begin{remark}
    The results above are naturally organised around two complementary mechanisms. Reichel's adaptation of Serrin's moving-plane method treats the case $\Gamma=0$ and extends to the regime $\Gamma<0$. By contrast, Agostiniani--Mazzieri's sharp boundary inequality for capacitary potentials yields rigidity at $\Gamma=N-2$ and continues to underpin the argument for $\Gamma>N-2$.
\end{remark}

\begin{remark}
    Following \cite{MR3383177}, we mention an interesting geometric reformulation of
    \eqref{eq:normalized-problem} for $\Gamma \neq 0$. Define the conformal metric
    \[
        g_\Gamma:=u^{2/\Gamma}\,\delta
        \qquad\text{in }\R^N\setminus\overline\Omega,
    \]
    where $\delta$ is the standard Euclidean metric on $\R^N$. Since $u=1$ on $\Sigma$, the conformal change formula for the mean curvature gives
    \[
        H_{g_\Gamma}
        =
        H+\frac{N-1}{\Gamma}\partial_\nu u
        =
        \frac{N-1}{\Gamma}\bigl(\Gamma-(N-2)\bigr),
    \]
    so $\Sigma$ has constant mean curvature in $(\R^N\setminus\overline\Omega,g_\Gamma)$.
    Moreover, a direct computation yields the scalar curvature as
    \[
        R_{g_\Gamma}
        =
        \frac{N-1}{\Gamma^2}\bigl(2\Gamma-(N-2)\bigr)
        u^{-2-\frac2\Gamma}|\nabla u|^2.
    \]
    Hence
    \[
        R_{g_\Gamma}\ge 0
        \qquad\Longleftrightarrow\qquad
        \Gamma\ge \frac{N-2}{2}.
    \]
    Thus $\frac{N-2}{2}$ appears as the scalar-curvature sign-change threshold for the conformal metric $g_\Gamma$. However, for $N\ge 3$,
    \[
        \frac{N-2}{2} \le \Gamma_2^{(N)}=\frac{(N-1)(N-2)}{N+1},
    \]
    so this threshold lies below the first bifurcation value. In this sense,
    $\frac{N-2}{2}$ is a natural algebraic threshold, but not the correct
    global rigidity threshold.
    We emphasise that for $\Gamma = N-2$ we have $H_{g_\Gamma} = 0$, so the boundary becomes a minimal hypersurface in the conformal metric, while $H_{g_\Gamma} > 0$ for $\Gamma > N-2$.

    We refer to \cite{MR2393076,MR4910987} for the role of nonnegative scalar curvature in capacity/rigidity inequalities for harmonic functions.
\end{remark}

\section{Rigidity for \texorpdfstring{$\Gamma=N-2$}{Gamma = N-2}}

In this section we prove Theorem~\ref{thm:critical-rigidity}. The key input is a sharp
boundary inequality of Agostiniani and Mazzieri for capacitary potentials. We recall the statement for the reader's convenience.

\begin{proposition}[Agostiniani--Mazzieri]\label{prop:AM-boundary}
    Let $D\subset\R^N$, $N\ge 3$, be a bounded smooth domain, and let $v\in \C^2(\R^N\setminus\overline D)\cap \C^1(\R^N\setminus D)$ solve
    \[
        \Delta v=0
        \quad\text{in }\R^N\setminus\overline D,
        \qquad
        v=1
        \quad\text{on }\partial D,
        \qquad
        v(x)\to0
        \quad\text{as }|x|\to\infty.
    \]
    Let $H_D$ denote the mean curvature of $\partial D$ with respect to the outer unit
    normal of $D$. Then, for every
    \[
        p\ge 2-\frac1{N-1},
    \]
    one has
    \begin{equation}\label{eq:AM-boundary}
        \|\partial_\nu v\|_{\L^p(\partial D)}
        \le
        \frac{N-2}{N-1}\|H_D\|_{\L^p(\partial D)}.
    \end{equation}
    Moreover, if equality holds in \eqref{eq:AM-boundary} for one admissible exponent $p$,
    then $v$ is rotationally symmetric with respect to some point $x_0\in\R^N$; that is,
    \[
        v(x)=f(|x-x_0|)
    \]
    for some one-variable function $f$.
\end{proposition}

\begin{proof}
    This is precisely \cite[Theorem~4.1]{MR4037467}.
\end{proof}

We write
\[
    H:=-(N-1)\Hn
\]
for the mean curvature of $\Sigma$ with respect to the outer unit normal of $\Omega$. At the value $\Gamma=N-2$, the boundary condition in \eqref{eq:normalized-problem}
becomes
\begin{equation}\label{eq:critical-boundary}
    \partial_\nu u=-\frac{N-2}{N-1}H.
\end{equation}

\begin{proof}[Proof of Theorem~\ref{thm:critical-rigidity}]
    By Remark~\ref{rem:regularity}, any $\C^{1,1}$ solution of
    \eqref{eq:normalized-problem} with $\Gamma=N-2\neq 0$ is smooth. Set
    \[
        E:=\R^N\setminus\overline\Omega,
        \qquad
        \Sigma:=\partial\Omega.
    \]

    We first claim that $E$ is connected. Suppose by contradiction that $U$ is a
    bounded connected component of $E$. Then $u$ is harmonic in $U$ and satisfies
    $u=1$ on $\partial U$. By the maximum principle,
    \[
        u= 1
        \qquad\text{in }U.
    \]
    Hence $\partial_\nu u=0$ on $\partial U$. By \eqref{eq:critical-boundary} we get
    $H=0$ on $\partial U$, so $\partial U$ is a compact minimal hypersurface without
    boundary in $\R^N$. This is impossible; see \cite[Corollary~3.7]{MR3969932}.
    Therefore $E$ is connected.

    Since $\Gamma=N-2$, we have on $\Sigma$
    \[
        \partial_\nu u=-\frac{N-2}{N-1}H.
    \]
    Therefore
    \[
        \|\partial_\nu u\|_{\L^2(\Sigma)}
        =
        \frac{N-2}{N-1}\|H\|_{\L^2(\Sigma)}.
    \]
    Thus equality holds in \eqref{eq:AM-boundary} for the admissible exponent $p=2$.
    Proposition~\ref{prop:AM-boundary} yields the existence of a point $x_0\in\R^N$
    and a function $f$ such that
    \[
        u(x)=f(|x-x_0|)
        \qquad\text{for all }x\in E.
    \]

    Since $E$ is the unbounded connected exterior domain, $u=1$ on $\Sigma$, and
    $u(x)\to 0$ as $|x|\to\infty$, the maximum principle gives
    \[
        0<u<1
        \qquad\text{in }E.
    \]
    We may now apply Lemma~\ref{lem:radial-implies-ball}. It follows that there exists
    $R>0$ such that
    \[
        \Omega=B_R(x_0).
    \]

    The volume constraint $|\Omega|=\omega_N$ implies $R=1$. Finally, the unique
    harmonic function in $\R^N\setminus\overline{B_1(x_0)}$ with boundary value $1$
    on $\partial B_1(x_0)$ and vanishing at infinity is
    \[
        u(x)=|x-x_0|^{2-N}.
    \]
\end{proof}

The following lemma is known, but we could not find a citable proof, so we include the argument for the reader's convenience.

\begin{lemma}\label{lem:radial-implies-ball}
    Let $N\ge 2$, let $\emptyset \neq \Omega\subset\R^N$ be bounded, and set
    \[
        E:=\R^N\setminus\overline\Omega.
    \]
    Assume that $E$ is connected. Suppose that for some $x_0\in\R^N$ and some
    function $f$ one has
    \[
        u(x)=f(|x-x_0|)
        \qquad\text{for all }x\in E,
    \]
    and that
    \[
        u=1\quad\text{on }\partial\Omega,
        \qquad
        0<u<1\quad\text{in }E.
    \]
    Then there exists $R>0$ such that $\Omega=B_R(x_0)$.
\end{lemma}

\begin{proof}
    Define
    \[
        I:=\{|x-x_0|:x\in E\}\subset[0,\infty).
    \]

    We first claim that for every $r\in I$,
    \[
        \partial B_r(x_0)\subset E.
    \]
    Indeed, let $r\in I$. Then $\partial B_r(x_0)\cap E\neq\emptyset$. If
    $\partial B_r(x_0)\not\subset E$, then $\partial B_r(x_0)$ meets $\overline\Omega$. Since
    $\partial B_r(x_0)$ is connected for $N\ge2$, it must then meet $\partial\Omega$.
    Thus there exist
    \[
        y\in \partial B_r(x_0)\cap E,
        \qquad
        z\in \partial B_r(x_0)\cap\partial\Omega.
    \]
    But then
    \[
        u(y)=f(r)\in(0,1),
        \qquad
        u(z)=1,
    \]
    a contradiction. This proves the claim.

    Since $\partial B_r(x_0)\subset E$ and $E$ is open, compactness of $\partial B_r(x_0)$ gives
    $\varepsilon_r>0$ such that
    \[
        \partial B_s(x_0)\subset E
        \qquad\text{whenever }|s-r|<\varepsilon_r.
    \]
    Hence $I$ is open.

    On the other hand, $I$ is connected because it is the image of the connected set
    $E$ under the continuous map $x\mapsto |x-x_0|$. Since $E$ is unbounded,
    $I$ is unbounded above. Therefore
    \[
        I=(R,\infty)
    \]
    for some $R\ge0$. Moreover, $R =0$ is impossible as $\Omega \neq \emptyset$.

    Moreover, we have
    \[
        E=\bigcup_{r\in I}\partial B_r(x_0)
        =\{x\in\R^N:|x-x_0|>R\}.
    \]
    Thus
    \[
        \partial\Omega=\partial E=\partial B_R(x_0),
    \]
    and therefore $\Omega=B_R(x_0)$.
\end{proof}

\section{Rigidity in the regime \texorpdfstring{$\Gamma > N-2$}{Gamma > N-2}}

In this section we prove Theorem~\ref{thm:star-outer-rigidity}. Let $\Omega\subset\R^N$ be a bounded star-shaped $\C^{1,1}$ domain such that $|\Omega|=\omega_N$, and let
\[
    u\in \C^2(\R^N\setminus\overline\Omega)\cap \C^1(\R^N\setminus\Omega)
\]
solve \eqref{eq:normalized-problem}. Then Remark~\ref{rem:regularity} implies that $\Sigma$ and $u$ are smooth. Since the problem is translation invariant, we may translate $\Omega$ and assume that the star centre is the origin. Set
\[
    \Sigma=\partial\Omega.
\]
Since $\Omega$ is star-shaped, we have
\begin{equation}\label{eq:p-positive}
    x\cdot\nu \ge 0
    \qquad\text{on }\Sigma
\end{equation}
by \cite[Lemma on p. 515]{MR2597943}.

We first claim that $\R^N\setminus\overline\Omega$ is connected. Indeed, if $U$ were a bounded connected component of $\R^N\setminus\overline\Omega$, then along $\partial U$ the outer unit normal of $\Omega$ is $-\nu_U$, where $\nu_U$ is the outer unit normal of $U$. Hence, by the divergence theorem,
\[
    \int_{\partial U} x\cdot\nu\dd S
    =
    -\int_{\partial U} x\cdot\nu_U\dd S
    =
    -N|U|
    <0,
\]
which contradicts \eqref{eq:p-positive}. Therefore $\R^N\setminus\overline\Omega$ is connected.

We again pass from the normalised mean curvature $\Hn$ to the usual mean curvature
\[
    H:=-(N-1)\Hn,
\]
and we write
\[
    \A:=A-\frac{H}{N-1}g
\]
for the trace-free second fundamental form of $\Sigma$, where $g$ denotes the induced metric on $\Sigma$. The boundary condition in \eqref{eq:normalized-problem} becomes
\begin{equation}\label{eq:qH}
    \partial_\nu u=-\frac{\Gamma}{N-1}H+\Gamma-(N-2).
\end{equation}
Since $u=1$ on $\Sigma$ and $u\to0$ at infinity, the maximum principle gives $0<u<1$ in $\R^N\setminus\overline\Omega$, and the Hopf lemma yields
\begin{equation}\label{eq:q-negative}
    \partial_\nu u<0
    \qquad\text{on }\Sigma.
\end{equation}

The next lemma collects the analytic identities that follow from the harmonicity of $u$.

\begin{lemma}\label{lem:pohozaev}
    Let $u$ solve \eqref{eq:normalized-problem}. Then
    \begin{equation}\label{eq:energy}
        \int_{\R^N\setminus\overline\Omega}|\nabla u|^2\dd x=-\int_\Sigma \partial_\nu u\dd S,
    \end{equation}
    and
    \begin{equation}\label{eq:pohozaev}
        \int_\Sigma (x\cdot\nu)(\partial_\nu u)^2\dd S=-(N-2)\int_\Sigma \partial_\nu u\dd S.
    \end{equation}
\end{lemma}

\begin{proof}
    Since $u$ is harmonic in the exterior of a bounded set and vanishes at infinity, it is the capacitary potential of $\Omega$. In particular,
    \[
        u(x)=O(|x|^{2-N}),
        \qquad
        |\nabla u(x)|=O(|x|^{1-N})
        \qquad\text{as }|x|\to\infty.
    \]
    Fix $R$ large enough that $\overline\Omega\subset B_R$, and set $D_R:=B_R\setminus\overline\Omega$. Integrating by parts on $D_R$ gives
    \[
        \int_{D_R}|\nabla u|^2\dd x
        =
        \int_{\partial B_R}u\partial_r u\dd S
        -
        \int_\Sigma u\,\partial_\nu u\dd S.
    \]
    Because $u=1$ on $\Sigma$ and the boundary term on $\partial B_R$ tends to zero as $R\to\infty$, we obtain \eqref{eq:energy}.

    For the Pohozaev identity we apply the divergence formula
    \[
        \operatorname{div}\Big(\frac12|\nabla u|^2x-(x\cdot\nabla u)\nabla u\Big)
        =
        \frac{N-2}{2}|\nabla u|^2
        \qquad\text{in }D_R,
    \]
    valid because $\Delta u=0$. Integrating over $D_R$, using the asymptotics at infinity, and letting $R\to\infty$, we get
    \[
        \frac{N-2}{2}\int_{\R^N\setminus\overline\Omega}|\nabla u|^2\dd x
        =
        \frac12\int_\Sigma (x\cdot\nu)(\partial_\nu u)^2\dd S.
    \]
    Now we insert \eqref{eq:energy} which gives \eqref{eq:pohozaev}.
\end{proof}

The next lemma records the geometric identities used below.

\begin{lemma}\label{lem:minkowski}
    Let $\Sigma=\partial\Omega$ be a closed orientable $\C^2$ hypersurface bounding a domain $\Omega\subset\R^N$.

    \smallskip
    \noindent\emph{(i) Minkowski identities.}
    One has
    \begin{align}
        \int_\Sigma x\cdot\nu\dd S
         & = N|\Omega|, \label{eq:minkowski0}                                                                 \\
        \int_\Sigma (x\cdot\nu)H\dd S
         & = (N-1)|\Sigma|, \label{eq:minkowski1}                                                             \\
        \int_\Sigma (x\cdot\nu)H^2\dd S
         & = (N-1)\int_\Sigma H\dd S+\frac{N-1}{N-2}\int_\Sigma (x\cdot\nu)|\A|^2\dd S. \label{eq:minkowski2}
    \end{align}

    \smallskip
    \noindent\emph{(ii) Weighted Cauchy--Schwarz estimate.}
    If, in addition, $x\cdot\nu \ge 0$ on $\Sigma$, then
    \begin{equation}\label{eq:U-lower-bound}
        \int_\Sigma (x\cdot\nu)|\A|^2\dd S
        \ge
        (N-2)\Big(\frac{(N-1)|\Sigma|^2}{N|\Omega|}-\int_\Sigma H\dd S\Big).
    \end{equation}
\end{lemma}

\begin{proof}
    \noindent\emph{(i)} The identity \eqref{eq:minkowski0} is the divergence theorem applied to the vector field $x$. The identity \eqref{eq:minkowski1} is the first Euclidean Minkowski formula; see \cite[\S5.5.2, formula~(5.70)]{MR4809877}. Moreover, \cite[\S5.5.2, formula~(5.79)]{MR4809877} gives
    \[
        \int_\Sigma (x\cdot\nu)\bigl(H^2-|A|^2\bigr)\dd S
        =
        (N-2)\int_\Sigma H\dd S.
    \]
    Using
    \[
        |A|^2=|\A|^2+\frac{H^2}{N-1},
    \]
    and rearranging then yields \eqref{eq:minkowski2}.

    \smallskip
    \noindent\emph{(ii)} Assume now that $x\cdot\nu\ge 0$ on $\Sigma$. By Cauchy--Schwarz with the nonnegative weight $x\cdot\nu$,
    \[
        \Big(\int_\Sigma (x\cdot\nu)H\dd S\Big)^2
        \le
        \Big(\int_\Sigma x\cdot\nu\dd S\Big)
        \Big(\int_\Sigma (x\cdot\nu)H^2\dd S\Big).
    \]
    Using \eqref{eq:minkowski0}--\eqref{eq:minkowski2} from part~(i), we obtain
    \[
        (N-1)^2|\Sigma|^2
        \le
        N|\Omega|
        \Bigl((N-1)\int_\Sigma H\dd S+\frac{N-1}{N-2}\int_\Sigma (x\cdot\nu)|\A|^2\dd S\Bigr),
    \]
    which is equivalent to \eqref{eq:U-lower-bound}.
\end{proof}

The Pohozaev identity and the Minkowski formulas combine into the following defect identity.

\begin{proposition}\label{prop:defect-identity}
    Set
    \[
        \lambda:=\Gamma-(N-2)>0.
    \]
    Then
    \begin{align}\label{eq:defect-identity}
         & \frac{\Gamma^2}{(N-1)(N-2)}\int_\Sigma (x\cdot\nu)|\A|^2\dd S
        \\
         & \quad +
        \frac{\Gamma\lambda}{N-1}\int_\Sigma H\dd S
        -
        \lambda(2\Gamma-(N-2))|\Sigma|
        +
        \lambda^2|\partial B_1|
        =0 \nonumber.
    \end{align}
\end{proposition}

\begin{proof}
    We recall from \eqref{eq:qH}:
    \[
        \partial_\nu u=-\frac{\Gamma}{N-1}H+\lambda.
    \]
    Expanding the left-hand side of \eqref{eq:pohozaev} and using \eqref{eq:minkowski0}--\eqref{eq:minkowski2}, we get
    \begin{align*}
        \int_\Sigma (x\cdot\nu)(\partial_\nu u)^2\dd S
         & =
        \frac{\Gamma^2}{(N-1)^2}\int_\Sigma (x\cdot\nu)H^2\dd S
        -
        \frac{2\Gamma\lambda}{N-1}\int_\Sigma (x\cdot\nu)H\dd S
        \\
         & \quad+
        \lambda^2\int_\Sigma x\cdot\nu\dd S                           \\
         & =
        \frac{\Gamma^2}{N-1}\int_\Sigma H\dd S
        +
        \frac{\Gamma^2}{(N-1)(N-2)}\int_\Sigma (x\cdot\nu)|\A|^2\dd S \\
         & \qquad
        -
        2\Gamma\lambda |\Sigma|
        +
        \lambda^2|\partial B_1|.
    \end{align*}
    On the other hand,
    \begin{align*}
        -(N-2)\int_\Sigma \partial_\nu u\dd S
         & =
        -(N-2)\Bigl(-\frac{\Gamma}{N-1}\int_\Sigma H\dd S+\lambda |\Sigma|\Bigr) \\
         & =
        \frac{(N-2)\Gamma}{N-1}\int_\Sigma H\dd S-(N-2)\lambda |\Sigma|.
    \end{align*}
    Inserting these identities into \eqref{eq:pohozaev} gives \eqref{eq:defect-identity}.
\end{proof}

\begin{proof}[Proof of Theorem~\ref{thm:star-outer-rigidity}]
    Set
    \[
        \lambda:=\Gamma-(N-2)>0.
    \]

    Since $\R^N \setminus \overline \Omega$ is connected, $u$ is the capacitary potential of $\Omega$, so Proposition~\ref{prop:AM-boundary} applies. Since
    \[
        -\partial_\nu u
        =
        \frac{\Gamma}{N-1}H-\lambda
        \qquad\text{on }\Sigma,
    \]
    the case $p=2$ in Proposition~\ref{prop:AM-boundary} gives
    \[
        \left\|\frac{\Gamma}{N-1}H-\lambda\right\|_{\L^2(\Sigma)}
        =
        \|\partial_\nu u\|_{\L^2(\Sigma)}
        \le
        \frac{N-2}{N-1}\|H\|_{\L^2(\Sigma)}.
    \]
    On the other hand, by the triangle inequality in $\L^2(\Sigma)$,
    \[
        \frac{\Gamma}{N-1}\|H\|_{\L^2(\Sigma)}-\lambda |\Sigma|^{1/2}
        \le
        \left\|\frac{\Gamma}{N-1}H-\lambda\right\|_{\L^2(\Sigma)}.
    \]
    Combining the last two estimates and using $\lambda=\Gamma-(N-2)$, we find
    \[
        \frac{\lambda}{N-1}\|H\|_{\L^2(\Sigma)}
        \le
        \lambda |\Sigma|^{1/2}.
    \]
    Since $\lambda>0$, it follows that
    \[
        \|H\|_{\L^2(\Sigma)}
        \le
        (N-1)|\Sigma|^{1/2}.
    \]
    Hence, by H\"older's inequality,
    \begin{equation}\label{eq:star-outer-M-bound}
        \int_\Sigma H\dd S
        \le
        |\Sigma|^{1/2}\|H\|_{\L^2(\Sigma)}
        \le
        (N-1)|\Sigma|.
    \end{equation}

    Using \eqref{eq:U-lower-bound} in \eqref{eq:defect-identity} and $\Gamma^2/((N-1)(N-2))>0$, we obtain
    \[
        0
        \ge
        \frac{\Gamma^2|\Sigma|^2}{|\partial B_1|}
        -
        \frac{\Gamma(N-2)}{N-1}\int_\Sigma H\dd S
        -
        \lambda(2\Gamma-(N-2))|\Sigma|
        +
        \lambda^2|\partial B_1|.
    \]
    Applying \eqref{eq:star-outer-M-bound} gives
    \[
        0
        \ge
        \frac{\Gamma^2|\Sigma|^2}{|\partial B_1|}
        -
        \Gamma(N-2)|\Sigma|
        -
        \lambda(2\Gamma-(N-2))|\Sigma|
        +
        \lambda^2|\partial B_1|.
    \]
    Since $\lambda=\Gamma-(N-2)$, this simplifies to
    \[
        0
        \ge
        \frac{|\Sigma|-|\partial B_1|}{|\partial B_1|}
        \bigl(\Gamma^2|\Sigma|-\lambda^2|\partial B_1|\bigr).
    \]
    By the isoperimetric inequality,
    \[
        |\Sigma|\ge |\partial B_1|.
    \]
    On the other hand,
    \[
        \Gamma^2|\Sigma|-\lambda^2|\partial B_1|
        \ge
        (\Gamma^2-\lambda^2)|\partial B_1|
        =
        (N-2)(2\Gamma-(N-2))|\partial B_1|
        >0.
    \]
    Therefore $|\Sigma|=|\partial B_1|$. Equality in the isoperimetric inequality implies that $\Omega$ is a ball, and the volume condition $|\Omega|=\omega_N$ forces its radius to be one.
\end{proof}

\section{Rigidity in the regime \texorpdfstring{$\Gamma \le 0$}{Gamma <= 0}}

In the case $\Gamma = 0$ the result is contained in \cite{MR1463801}.
\begin{proof}[Proof of Theorem \ref{thm:gamma-zero}]
    Any bounded connected component of $\R^N\setminus\overline\Omega$ would force
    $u= 1$ there by the maximum principle, contradicting
    $\partial_\nu u=-(N-2)$ on its boundary. Hence
    $\R^N\setminus\overline\Omega$ is connected, and therefore
    \[
        0<u<1
        \qquad\text{in }\R^N\setminus\overline\Omega.
    \]
    Since $\Gamma=0$, the boundary condition in \eqref{eq:normalized-problem} is
    \[
        \partial_\nu u=-(N-2)
        \qquad\text{on }\partial\Omega.
    \]
    Thus Reichel's theorem \cite[Theorem~1]{MR1463801} applies and yields that
    $\Omega$ is a ball.
\end{proof}

To treat $\Gamma <0$ we revisit Reichel's adaptation \cite{MR1463801} of Serrin's moving plane method \cite{MR333220} to the exterior case. We refer to \cite{MR1751289} for a textbook introduction to Serrin's moving plane method. The idea is that when $\Gamma < 0$, the sign of the mean curvature is compatible with the conclusions of Hopf's lemma and Serrin's corner lemma.

\begin{proof}
    By Remark~\ref{rem:regularity}, any $\C^{1,1}$ solution of
    \eqref{eq:normalized-problem} with $\Gamma < 0$ is smooth. Set
    \[
        E:=\R^N\setminus\overline\Omega,
        \qquad
        \Sigma:=\partial\Omega.
    \]
    We again write
    \[
        H:=-(N-1)\Hn,
    \]
    so that the boundary condition becomes
    \begin{equation}\label{eq:negative-qH}
        \partial_\nu u=-\frac{\Gamma}{N-1}H+\Gamma-(N-2).
    \end{equation}

    We first claim that $E$ is connected. Suppose by contradiction that $U$ is a bounded connected component of $E$. Then $u$ is harmonic in $U$ and satisfies $u=1$ on $\partial U$, hence by the maximum principle
    \[
        u = 1
        \qquad\text{in }U.
    \]
    Therefore $\partial_\nu u=0$ on $\partial U$, and \eqref{eq:negative-qH} yields
    \[
        H=\frac{N-1}{\Gamma}\bigl(\Gamma-(N-2)\bigr)>0
        \qquad\text{on }\partial U,
    \]
    where $H$ is computed with respect to the outer normal $\nu$ of $\Omega$.
    Let $\nu_U:=-\nu$ be the outer unit normal of $U$, and let $H_U:=-H$ be the mean curvature of $\partial U$ with respect to $\nu_U$. Then $H_U<0$ is constant on $\partial U$. The first Minkowski formula \eqref{eq:minkowski1} of Lemma \ref{lem:minkowski} applied to the bounded domain $U$ gives
    \[
        (N-1)|\partial U|
        =
        \int_{\partial U}(x\cdot\nu_U)H_U\dd S
        =
        H_U\int_{\partial U}x\cdot\nu_U\dd S
        =
        NH_U|U|<0,
    \]
    a contradiction. Thus $E$ is connected.

    Since $u=1$ on $\Sigma$ and $u(x)\to0$ as $|x|\to\infty$, the maximum principle gives
    \[
        0<u<1
        \qquad\text{in }E,
    \]
    and the Hopf lemma yields
    \begin{equation}\label{eq:negative-hopf}
        \partial_\nu u<0
        \qquad\text{on }\Sigma.
    \end{equation}
    \medskip

    We now perform the method of moving planes, following Reichel's proof of the exterior Serrin problem \cite{MR1463801}. We also refer to \cite{MR1463801} for very nice illustrations. Fix a direction and, after a rotation, assume it is $e_1$. For $\lambda\in\R$ set
    \[
        T_\lambda:=\{x_1=\lambda\},
        \qquad
        H_\lambda^-:=\{x_1<\lambda\},
        \qquad
        x^\lambda:=(2\lambda-x_1,x_2,\dots,x_N).
    \]
    Let
    \[
        \Omega_\lambda:=\Omega\cap\{x_1>\lambda\},
        \qquad
        \Omega_\lambda':=\{x^\lambda:x\in\Omega_\lambda\}.
    \]
    Since $\Omega$ is bounded, the set
    \[
        \Lambda
        :=
        \bigl\{\lambda\in\R:\Omega_\mu'\subset\Omega
        \text{ for every }\mu\ge\lambda\bigr\}
    \]
    is nonempty.
    Define
    \[
        \lambda_*:=\inf\Lambda.
    \]
    We note that
    \[
        \Omega_{\lambda_*}'\subset\Omega.
    \]
    Indeed, let $x\in\Omega_{\lambda_*}$. Choose $\lambda_n\downarrow\lambda_*$ with $\lambda_n<x_1$. Since $\lambda_n>\lambda_*$, we have $\lambda_n\in\Lambda$, hence $x^{\lambda_n}\in\Omega$ for every $n$. Passing to the limit gives $x^{\lambda_*}\in\overline{\Omega}$. Therefore $\Omega_{\lambda_*}'\subset\overline{\Omega}$, and since $\Omega_{\lambda_*}'$ is open while $\Omega$ is the interior of $\overline{\Omega}$, it follows that $\Omega_{\lambda_*}'\subset\Omega$.

    For $\lambda>\lambda_*$ we set
    \[
        E_\lambda:=E\cap H_\lambda^-,
        \qquad
        w_\lambda(x):=u(x)-u(x^\lambda),
        \qquad x\in E_\lambda.
    \]
    Because $\Omega_\lambda'\subset\Omega$, every $x\in E_\lambda$ satisfies $x^\lambda\in E$, so $w_\lambda$ is well-defined and harmonic in $E_\lambda$. Moreover,
    \[
        w_\lambda=0
        \qquad\text{on }T_\lambda\cap\overline{E_\lambda},
    \]
    while on $\Sigma\cap H_\lambda^-$ we have
    \[
        w_\lambda(x)=1-u(x^\lambda).
    \]
    We claim that $x^\lambda\in E\cup\Sigma$. Indeed, if $x^\lambda\in\Omega$,
    then $x_1^\lambda>\lambda$, so $x^\lambda\in\Omega_\lambda$, and hence
    \[
        x=(x^\lambda)^\lambda\in\Omega_\lambda'\subset\Omega,
    \]
    contrary to $x\in\Sigma$. Thus $x^\lambda\notin\Omega$, that is,
    $x^\lambda\in E\cup\Sigma$. Since $0<u<1$ in $E$ and $u=1$ on $\Sigma$,
    it follows that $u(x^\lambda)\le1$, and therefore
    \[
        w_\lambda(x)\ge0.
    \]
    Finally, $w_\lambda(x)\to0$ as $|x|\to\infty$ with $x\in E_\lambda$. Hence the maximum principle yields
    \[
        w_\lambda\ge0
        \qquad\text{in }E_\lambda
        \qquad(\lambda>\lambda_*).
    \]
    Passing to the limit gives
    \[
        w_{\lambda_*}\ge0
        \qquad\text{in }E_{\lambda_*}.
    \]
    By the strong maximum principle, on each connected component $D$ of
    $E_{\lambda_*}$ either
    \[
        w_{\lambda_*}\equiv 0
        \qquad\text{or}\qquad
        w_{\lambda_*}>0
        \quad\text{in }D.
    \]

    We now use Reichel's Step~(I), adapted to the present notation: if
    $w_{\lambda_*}\equiv0$ on one connected component of $E_{\lambda_*}$,
    then $E$ is symmetric with respect to $T_{\lambda_*}$, and hence so is
    $\Omega$; see \cite[Step~(I), pp.~385--386]{MR1463801}. Therefore, unless
    $\Omega$ is already symmetric with respect to $T_{\lambda_*}$, we have
    \[
        w_{\lambda_*}>0
        \qquad\text{in every connected component of }E_{\lambda_*}.
    \]

    At the critical position, the geometric alternative of the
    moving-plane method applies: either the reflected cap $\Omega_{\lambda_*}'$
    is internally tangent to $\Sigma$ at some point $P\notin T_{\lambda_*}$,
    or the plane $T_{\lambda_*}$ meets $\Sigma$ orthogonally at some point
    $Q\in\Sigma\cap T_{\lambda_*}$; compare \cite[p.~384, Figure~2]{MR1463801}
    and \cite[Step~(VI), pp.~388--390]{MR1463801}.

    We now rule out both possibilities, following Reichel's final step, with
    the new input being the mean-curvature term in the boundary condition.

    \medskip
    We first treat the interior tangency case, following
    \cite[Step~(VI)(a), p.~388]{MR1463801}. Let $D_P$ be the connected
    component of $E_{\lambda_*}$ whose boundary contains $P$. By the
    previous paragraph,
    \[
        w_{\lambda_*}>0
        \qquad\text{in }D_P.
    \]
    Let
    \[
        \Sigma_{\lambda_*}':=\partial\Omega_{\lambda_*}'\setminus T_{\lambda_*}
    \]
    denote the reflected curved boundary.
    At the touching point $P$, the hypersurfaces $\Sigma$ and $\Sigma_{\lambda_*}'$ are tangent, have the same outer normal $\nu(P)$, and $\Sigma_{\lambda_*}'$ lies locally on the interior side of $\Sigma$. After translation and rotation we may assume
    \[
        P=0,
        \qquad
        \nu(P)=e_N,
    \]
    and that, in a neighbourhood of $0$, both hypersurfaces are graphs over $\R^{N-1}$:
    \[
        \Sigma=\{x_N=\varphi(x')\},
        \qquad
        \Sigma_{\lambda_*}'=\{x_N=\psi(x')\},
    \]
    with
    \[
        \varphi(0)=\psi(0)=0,
        \qquad
        \nabla\varphi(0)=\nabla\psi(0)=0,
        \qquad
        \varphi\ge\psi.
    \]
    Thus $\varphi-\psi$ has a local minimum at $0$, so
    \[
        D^2\varphi(0)-D^2\psi(0)\ge0
    \]
    as quadratic forms, and therefore
    \[
        \Delta\varphi(0)\ge\Delta\psi(0).
    \]
    Since for such graphs the normalised mean curvature with respect to the upward normal is
    \[
        \Hn=\frac1{N-1}\Delta\varphi
        \qquad\text{at points where }\nabla\varphi=0,
    \]
    we obtain
    \begin{equation}\label{eq:negative-curv-comp}
        \Hn_\Sigma(P)\ge \Hn_{\Sigma_{\lambda_*}'}(P).
    \end{equation}

    Now $w_{\lambda_*}(P)=0$, and since $\nu(P)$ points into $D_P$, the Hopf lemma gives
    \[
        \partial_\nu w_{\lambda_*}(P)>0.
    \]
    On the other hand, reflection preserves $\Hn$, and the reflected function
    \[
        u^{\lambda_*}(x):=u(x^{\lambda_*})
    \]
    satisfies the same overdetermined condition on $\Sigma_{\lambda_*}'$. Hence
    \begin{align*}
        \partial_\nu w_{\lambda_*}(P)
         & =
        \partial_\nu u(P)-\partial_\nu u^{\lambda_*}(P) \\
         & =
        \Gamma\bigl(\Hn_\Sigma(P)-\Hn_{\Sigma_{\lambda_*}'}(P)\bigr)
        \le 0
    \end{align*}
    by \eqref{eq:negative-curv-comp}, because $\Gamma<0$. This contradiction rules out the interior tangency case.

    \medskip
    It remains to treat the orthogonality case. We use the same local
    coordinate setup as in \cite[Step~(VI)(b), pp.~389--390]{MR1463801}. Let
    $D_Q$ denote the connected component of $E_{\lambda_*}$ whose boundary
    contains $Q$. Again,
    \[
        w_{\lambda_*}>0
        \qquad\text{in }D_Q.
    \]
    After translation and rotation we may assume
    \[
        \lambda_*=0,
        \qquad
        Q=0,
        \qquad
        T_{\lambda_*}=\{x_1=0\},
        \qquad
        \nu(Q)=e_N.
    \]
    Near $0$, write $\Sigma$ as a graph
    \[
        \Sigma=\{x_N=\varphi(x_1,\xi)\},
        \qquad
        \xi=(x_2,\dots,x_{N-1}),
    \]
    with
    \[
        \Omega=\{x_N<\varphi(x_1,\xi)\},
        \qquad
        \varphi(0)=0,
        \qquad
        \nabla\varphi(0)=0.
    \]
    Since the reflected cap is contained in $\Omega$ at the critical position, we have
    \[
        \varphi(x_1,\xi)\ge \varphi(-x_1,\xi)
        \qquad\text{for }x_1<0
    \]
    in a neighbourhood of $0$. Dividing by $2x_1<0$ and letting $x_1\uparrow0$, we obtain
    \[
        \partial_{x_1}\varphi(0,\xi)\le0
        \qquad\text{for $\xi$ near $0$}.
    \]
    Since $\partial_{x_1}\varphi(0,0)=0$, the function $\xi\mapsto\partial_{x_1}\varphi(0,\xi)$ has a local maximum at $\xi=0$. Hence
    \[
        \partial_{x_j}\partial_{x_1}\varphi(0)=0
        \qquad(2\le j\le N-1),
    \]
    and the matrix
    \[
        B:=\bigl(\partial_{x_j}\partial_{x_k}\partial_{x_1}\varphi(0)\bigr)_{2\le j,k\le N-1}
    \]
    is negative semidefinite. In particular,
    \[
        \operatorname{tr}B\le0.
    \]
    In particular, for
    \[
        h(t):=\varphi(t,0)-\varphi(-t,0)
    \]
    we have $h(t)\ge0$ for $t<0$, while
    \[
        h(t)=\frac13\,\partial_{x_1}^3\varphi(0)t^3+o(t^3)
        \qquad\text{as }t\to0.
    \]
    Since $t^3<0$ for $t<0$, it follows that
    \[
        \partial_{x_1}^3\varphi (0)\le0.
    \]

    At points where $\nabla\varphi=0$, the normalised mean curvature with respect to the upward normal is
    \[
        \Hn=\frac1{N-1}\operatorname{div}\!\left(\frac{\nabla\varphi}{\sqrt{1+|\nabla\varphi|^2}}\right),
    \]
    hence
    \[
        \partial_{x_1}\Hn(0)
        =
        \frac1{N-1}\Bigl(\varphi_{111}(0)+\operatorname{tr}B\Bigr)
        \le0.
    \]
    Define
    \[
        q(x_1,\xi):=\partial_\nu u\bigl(x_1,\xi,\varphi(x_1,\xi)\bigr).
    \]
    By the boundary condition,
    \[
        q=\Gamma\Hn+\Gamma-(N-2),
    \]
    so, because $\Gamma<0$,
    \[
        \partial_{x_1}q(0)=\Gamma\,\partial_{x_1}\Hn(0)\ge0.
    \]
    On the other hand, differentiating the boundary identity
    \[
        u\bigl(x_1,\xi,\varphi(x_1,\xi)\bigr)=1
    \]
    shows that
    \[
        \partial_{x_i}u(0)=0
        \qquad(1\le i\le N-1).
    \]
    Writing
    \[
        q
        =
        \frac{\partial_{x_N}u-\sum_{i=1}^{N-1}\partial_{x_i}u \partial_{x_i}\varphi}{\sqrt{1+|\nabla\varphi|^2}}
    \]
    on $\Sigma$, we therefore obtain
    \[
        \partial_{x_1}q(0)=\partial_{x_N}\partial_{x_1}u(0).
    \]
    Consequently,
    \[
        \partial_{x_N}\partial_{x_1} u(0)\ge0.
    \]

    Now we abbreviate $w:=w_{\lambda_*}$. Since reflection across $\{x_1=0\}$ flips $e_1$ and fixes $e_N$, we have
    \[
        \nabla w(0)=0,
        \qquad
        \partial_{x_1}^2w(0)=0,
        \qquad
        \partial_{x_N}w(0)=0,
        \qquad
        \partial_{x_N}\partial_{x_1} w(0)=2\partial_{x_N}\partial_{x_1}u(0)\ge0.
    \]
    Moreover, as $w(x)=u(x)-u(x^{\lambda_*})$ and here $\lambda_*=0$, we have $x^{\lambda_*}=(-x_1,x_2,\dots,x_N)$, so the reflection leaves the $x_N$-coordinate unchanged and therefore
    \[
        \partial_{x_N}^2w(x)=\partial_{x_N}^2u(x)-\partial_{x_N}^2u(x^{\lambda_*}).
    \]
    Evaluating at $Q=0\in T_{\lambda_*}$, where $0^{\lambda_*}=0$, yields
    \[
        \partial_{x_N}^2w(0)=\partial_{x_N}^2u(0)-\partial_{x_N}^2u(0)=0.
    \]

    Let
    \[
        s:=\frac{-e_1+e_N}{\sqrt2},
    \]
    which points into the corner $E_{\lambda_*}$ at $Q$. Then
    \[
        \partial_s w(0)=0
    \]
    and
    \[
        \partial_{ss}w(0)
        =
        \frac12\bigl(\partial_{x_1}^2w(0)+\partial_{x_N}^2w(0)-2\partial_{x_N}\partial_{x_1}w(0)\bigr)
        =
        -\partial_{x_N}\partial_{x_1}w(0)
        \le0.
    \]
    Since $w\ge0$ and $w>0$ in $D_Q$, $w$ is harmonic in $D_Q$, and $w(0)=0$, Serrin's corner lemma (\cite[Appendix]{MR1463801}) applied in the local corner domain $D_Q\cap B_r(0)$ for $r>0$ small and to the inward bisector direction $s$ yields a contradiction.

    Therefore the alternative that $w_{\lambda_*}>0$ in every connected
    component of $E_{\lambda_*}$ is impossible. Hence
    $w_{\lambda_*}\equiv 0$ on at least one connected component of
    $E_{\lambda_*}$. By the argument of Reichel's Step~(I), with left and
    right interchanged, it follows that $E$ is symmetric with respect to
    $T_{\lambda_*}$, and therefore so is $\Omega$.

    Since the direction $e_1$ was arbitrary, $\Omega$ is a ball. Finally,
    the volume constraint $|\Omega|=\omega_N$ implies that the radius is one.
    Hence
    \[
        \Omega=B_1(x_0)
        \qquad\text{for some }x_0\in\R^N.
    \]
\end{proof}

\section{The planar case}\label{sec:planar}

In dimension $N=2$ the decay condition at infinity is replaced by a logarithmic far-field condition. We consider
\begin{equation}\label{eq:original-problem-2d}
    \left\{
    \begin{aligned}
        -\Delta u          & = 0              &  & \text{in } \R^2\setminus\overline\Omega, \\
        u                  & = u_0            &  & \text{on } \Sigma:=\partial\Omega,       \\
        \partial_\nu u     & = \gamma \Hn + C &  & \text{on } \Sigma,                       \\
        u(x)+\alpha\log|x| & = O(1)           &  & \text{as } |x|\to\infty,
    \end{aligned}
    \right.
\end{equation}
where $\alpha \neq 0$ and $|\Omega|=\pi R_0^2$. Here $\Hn$ is the curvature of $\Sigma$ with the convention $\Hn_{\partial B_1} = -1$. For $\Omega=B_{R_0}$ the unique radial solution is
\[
    u_*(x)=u_0+\alpha\log\frac{R_0}{|x|},
\]
so the spherical compatibility value is
\[
    C=\frac{\gamma-\alpha}{R_0}.
\]
After the rescaling
\[
    x=R_0 y,
    \qquad
    u(x)=u_0+\alpha\big(v(y)-1\big),
    \qquad
    \widetilde\Omega=R_0^{-1}\Omega,
    \qquad
    \Gamma:=\frac\gamma\alpha,
\]
and dropping tildes, we arrive at the normalised planar problem
\begin{equation}\label{eq:normalized-problem-2d}
    \left\{
    \begin{aligned}
        -\Delta u      & =0                    &  & \text{in }\R^2\setminus\overline\Omega, \\
        u              & =1                    &  & \text{on }\Sigma,                       \\
        \partial_\nu u & = \Gamma \Hn+\Gamma-1 &  & \text{on }\Sigma,                       \\
        u(x)+\log|x|   & = O(1)                &  & \text{as }|x|\to\infty,
    \end{aligned}
    \right.
\end{equation}
with $|\Omega|=\pi$.

Following the argument in \cite[Section~2]{niebel2025globalrigiditytwodimensionalbubbles}, together with a short reduction excluding holes, we obtain global rigidity of the unit circle for bounded planar domains.

\begin{theorem} \label{thm:planar-jordan}
    Let $\Gamma\in\R$. Let $\Omega\subset\R^2$ be a bounded domain such that $|\Omega|=\pi$ and $\partial\Omega\in \C^{1,1}$. Assume that \eqref{eq:normalized-problem-2d} admits a solution $u \in \C^2(\R^2\setminus\overline\Omega)\cap \C^1(\R^2\setminus\Omega)$.
    Then $\Omega$ is a unit disk. Consequently,
    \[
        u(x)=1-\log|x-x_0|
        \qquad\text{for some }x_0\in\R^2,
    \]
    and $\partial\Omega=\partial B_1(x_0)$.
\end{theorem}

\begin{proof}
    Set
    \[
        \Sigma:=\partial\Omega.
    \]
    Since $\partial\Omega$ is a compact $\C^{1,1}$ one-dimensional manifold, it is a finite disjoint union of $\C^{1,1}$ Jordan curves. Because $\Omega$ is connected and bounded, we may write
    \[
        \partial\Omega=\Sigma_0\cup\Sigma_1\cup\cdots\cup\Sigma_{m-1},
    \]
    where $\Sigma_0$ is the outer boundary component and $\Sigma_1,\dots,\Sigma_{m-1}$ are the inner boundary components. We first claim that $m=1$.

    The boundary condition is
    \begin{equation}\label{eq:qH-2d}
        \partial_\nu u=\Gamma\Hn+\Gamma-1.
    \end{equation}

    Since $u$ is harmonic we have
    \[
        u(x)=-\log|x|+O(1),
        \qquad
        \partial_r u(x)=-\frac1{|x|}+O(|x|^{-2})
        \qquad\text{as }|x|\to\infty,
    \]
    and integrating $\Delta u=0$ over $D_R:=B_R\setminus\overline\Omega$ gives
    \[
        \int_{\partial B_R}\partial_r u\dd S-\int_{\partial\Omega}\partial_\nu u\dd S=0.
    \]
    Letting $R\to\infty$, we obtain
    \begin{equation}\label{eq:flux-2d}
        \int_{\partial\Omega}\partial_\nu u\dd S=-2\pi.
    \end{equation}

    Suppose by contradiction that $m\ge2$. For $j=1,\dots,m-1$, let $D_j$ be the bounded connected component of $\R^2\setminus\overline\Omega$ enclosed by $\Sigma_j$. Since $u$ is harmonic in $D_j$ and $u=1$ on $\Sigma_j$, the maximum principle gives
    \[
        u= 1
        \qquad\text{in }D_j.
    \]
    Hence
    \[
        \partial_\nu u=0
        \qquad\text{on }\Sigma_j.
    \]
    Integrating \eqref{eq:qH-2d} over $\Sigma_j$ and using
    \[
        \int_{\Sigma_j}\Hn\dd S=2\pi
    \]
    for an inner boundary component, we obtain
    \[
        0=2\pi\Gamma+(\Gamma-1)|\Sigma_j|.
    \]
    Since $|\Sigma_j|>0$, this is possible only if $0<\Gamma<1$.

    On the other hand, \eqref{eq:flux-2d} and the vanishing of $\partial_\nu u$ on the inner components imply
    \[
        \int_{\Sigma_0}\partial_\nu u\dd S=-2\pi.
    \]
    Moreover, for the outer boundary component,
    \[
        \int_{\Sigma_0}\Hn\dd S=-2\pi.
    \]
    Integrating \eqref{eq:qH-2d} over $\Sigma_0$ therefore yields
    \[
        -2\pi=-2\pi\Gamma+(\Gamma-1)|\Sigma_0|.
    \]
    Since $0<\Gamma<1$, it follows that
    \[
        |\Sigma_0|=2\pi.
    \]

    Let $G$ be the bounded Jordan domain enclosed by $\Sigma_0$. By the isoperimetric inequality,
    \[
        |G|\le\frac{|\Sigma_0|^2}{4\pi}=\pi.
    \]
    On the other hand,
    \[
        |G|
        =
        |\Omega|+\sum_{j=1}^{m-1}|D_j|
        =
        \pi+\sum_{j=1}^{m-1}|D_j|
        >
        \pi,
    \]
    a contradiction. Thus $m=1$. Therefore $\Sigma$ is a $\C^{1,1}$ Jordan curve, and $\R^2\setminus\overline\Omega$ is connected.

    By translation invariance we may assume that $0\in\Omega$.

    Since $u=1$ on $\Sigma$ and $u(x)\to-\infty$ as $|x|\to\infty$, the maximum principle gives $u<1$ in $\R^2\setminus\overline\Omega$, and the Hopf lemma yields
    \begin{equation}\label{eq:q-negative-2d}
        \partial_\nu u<0\qquad\text{on }\Sigma.
    \end{equation}

    If $\Gamma=0$, then $\partial_\nu u=-1$ on $\Sigma$, so by
    \eqref{eq:flux-2d} we get $|\Sigma|=2\pi$, and the isoperimetric inequality
    implies that $\Omega$ is a unit disk.

    Assume now that $\Gamma\neq0$. By Remark~\ref{rem:regularity} (which applies verbatim when $N = 2$),
    $\Sigma$ and $u$ are smooth.

    Integrating \eqref{eq:qH-2d} over $\Sigma$ and using the turning-angle formula
    \[
        \int_\Sigma \Hn\dd S=-2\pi
    \]
    yields
    \[
        -2\pi
        =
        \Gamma\int_\Sigma \Hn\dd S+(\Gamma-1)|\Sigma|
        =
        -2\pi\Gamma+(\Gamma-1)|\Sigma|.
    \]
    If $\Gamma\neq1$, $(\Gamma-1)(|\Sigma|-2\pi)=0$ implies $|\Sigma|=2\pi$. Since $|\Omega|=\pi$, the isoperimetric inequality implies that $\Omega$ is a unit disk.

    It remains to consider the borderline case $\Gamma=1$. Then \eqref{eq:qH-2d} gives $\partial_\nu u=\Hn$, so \eqref{eq:q-negative-2d} implies
    \[
        \Hn<0\qquad\text{on }\Sigma.
    \]
    Thus $\Sigma$ is a smooth strictly convex Jordan curve by \cite[Theorem 2.2.15]{MR2664879}. In particular,
    \[
        x\cdot\nu>0\qquad\text{on }\Sigma.
    \]

    We now use the planar Pohozaev identity
    \begin{equation}\label{eq:pohozaev-2d}
        \int_\Sigma (x\cdot\nu)(\partial_\nu u)^2\dd S=2\pi.
    \end{equation}
    Indeed, for
    \[
        X:=\frac12|\nabla u|^2x-(x\cdot\nabla u)\nabla u
    \]
    one has $\operatorname{div}X=0$ in $D_R$. On $\Sigma$, where the outer normal of $D_R$ is $-\nu$, one computes
    \[
        X\cdot(-\nu)=\frac12\,(x\cdot\nu)(\partial_\nu u)^2.
    \]
    On $\partial B_R$,
    \[
        X\cdot \frac{x}{|x|}=-\frac1{2R}+O(R^{-2}),
    \]
    hence
    \[
        \int_{\partial B_R}X\cdot\frac{x}{|x|}\dd S\to-\pi.
    \]
    Letting $R\to\infty$ gives \eqref{eq:pohozaev-2d}.

    We also have the planar Minkowski formulas
    \begin{equation}\label{eq:minkowski0-2d}
        \int_\Sigma x\cdot\nu\dd S=2|\Omega|=2\pi,
    \end{equation}
    \begin{equation}\label{eq:minkowski1-2d}
        -\int_\Sigma (x\cdot\nu)\Hn\dd S=|\Sigma|.
    \end{equation}
    Since now $\partial_\nu u=\Hn$, \eqref{eq:pohozaev-2d} becomes
    \[
        2\pi=\int_\Sigma (x\cdot\nu)\Hn^2\dd S.
    \]
    Therefore, by Cauchy--Schwarz with weight $x\cdot\nu$,
    \begin{align*}
        |\Sigma|^2
         & =
        \left(-\int_\Sigma (x\cdot\nu)\Hn\dd S\right)^2
        \le
        \left(\int_\Sigma x\cdot\nu\dd S\right)
        \left(\int_\Sigma (x\cdot\nu)\Hn^2\dd S\right) \\
         & =
        (2\pi)(2\pi)=4\pi^2.
    \end{align*}
    Hence $|\Sigma|\le2\pi$. The isoperimetric inequality gives $|\Sigma|\ge2\pi$, so $|\Sigma|=2\pi$, and again $\Omega$ is a unit disk.

    Finally, once $\Omega$ is a unit disk, the unique harmonic solution to \eqref{eq:normalized-problem-2d} is, up to translation, $u(x)=1-\log|x-x_0|$.
\end{proof}

\begin{remark}\label{rem:planar-linearization}
    Formally, let
    \[
        \Sigma_\varepsilon
        =
        \bigl\{(1+\varepsilon f(\theta))e^{i\theta}:\theta\in\mathbb S^1\bigr\},
        \qquad
        u_\varepsilon(r,\theta)
        =
        1-\log r+\varepsilon v(r,\theta)+o(\varepsilon).
    \]
    Then $v$ solves
    \[
        -\Delta v=0
        \quad\text{in }\R^2\setminus\overline{B_1},
        \qquad
        v=f
        \quad\text{on }\partial B_1,
        \qquad
        v=O(1)
        \quad\text{as }|x|\to\infty,
    \]
    and the linearised overdetermined condition reads
    \[
        \partial_r v(1,\theta)+f(\theta)
        =
        \Gamma\bigl(f_{\theta\theta}(\theta)+f(\theta)\bigr).
    \]
    Equivalently, if $\mathcal E(f)$ denotes the bounded exterior harmonic extension of $f$, then
    \[
        \mathcal L_\Gamma f
        :=
        \partial_r\mathcal E(f)\big|_{r=1}+f-\Gamma(f_{\theta\theta}+f)=0.
    \]
    If one sets $\Lambda f:=-\partial_r\mathcal E(f)|_{r=1}$, then
    \[
        \mathcal L_\Gamma = \mathrm{Id}-\Lambda-\Gamma(\partial_{\theta\theta}+I).
    \]

    Writing
    \[
        f(\theta)
        =
        a_0+\sum_{\ell\ge1}\bigl(a_\ell\cos(\ell\theta)+b_\ell\sin(\ell\theta)\bigr),
    \]
    one finds
    \[
        \mathcal L_\Gamma\bigl(\cos(\ell\theta)\bigr)
        =
        \mu_\ell(\Gamma)\cos(\ell\theta),
        \qquad
        \mathcal L_\Gamma\bigl(\sin(\ell\theta)\bigr)
        =
        \mu_\ell(\Gamma)\sin(\ell\theta),
    \]
    with
    \[
        \mu_\ell(\Gamma)=(1-\ell)\bigl(1-\Gamma(\ell+1)\bigr)
        =(\ell-1)\bigl(\Gamma(\ell+1)-1\bigr).
    \]
    Here $\ell=0$ corresponds to scaling, while $\ell=1$ corresponds to translations.

    The tangent space to the area constraint $|\Omega|=\pi$ at the unit disk is
    \[
        \left\{f:\int_0^{2\pi} f(\theta)\,d\theta=0\right\},
    \]
    so the constant mode is removed. Hence, on zero-mean perturbations,
    \[
        \ker \mathcal L_\Gamma
        =
        \operatorname{span}\{\cos\theta,\sin\theta\}
    \]
    for $\Gamma\notin\{\frac1{\ell+1}:\ell\ge2\}$, whereas at the resonant values
    \[
        \Gamma_\ell=\frac1{\ell+1},\qquad \ell\ge2,
    \]
    the modes $\cos(\ell\theta)$ and $\sin(\ell\theta)$ also belong to the linear kernel.

    These additional kernel directions do not generate nearby
    area-preserving branches. Indeed, a second-order ansatz
    \[
        R(\theta)=1+\varepsilon\cos(\ell\theta)
        +\varepsilon^2\bigl(a_0+a_2\cos(2\ell\theta)\bigr)+O(\varepsilon^3)
    \]
    yields $a_0=-\frac{\ell^2}{4}$
    from the constant Fourier mode of the overdetermined condition, whereas
    area preservation imposes $a_0=-\frac14$.
    These are compatible only for $\ell=1$. Thus no nontrivial $\ell$-fold
    branch exists for $\ell\ge2$, in agreement with Theorem~\ref{thm:planar-jordan}.
\end{remark}

Let us further deduce a topological identity for multiply connected boundaries.

\begin{lemma}\label{lem:planar-topological}
    Let $\Omega\subset\R^2$ be a bounded set of class $\C^{1,1}$, not necessarily connected, and assume that \eqref{eq:normalized-problem-2d} admits a solution $u \in \C^2(\R^2\setminus\overline\Omega)\cap \C^1(\R^2\setminus\Omega)$. Then
    \begin{equation}\label{eq:planar-topological}
        -2\pi=-2\pi\Gamma\chi(\Omega)+(\Gamma-1)|\partial\Omega|.
    \end{equation}
    Equivalently,
    \[
        (\Gamma-1)|\partial\Omega|=2\pi\bigl(\Gamma\chi(\Omega)-1\bigr).
    \]
\end{lemma}

\begin{proof}
    The boundary condition is
    \[
        \partial_\nu u=\Gamma\Hn+\Gamma-1
        \qquad\text{on }\partial\Omega.
    \]
    The flux identity \eqref{eq:flux-2d} remains valid for every bounded $\C^{1,1}$ domain, so
    \[
        \int_{\partial\Omega} \partial_\nu u\dd S=-2\pi.
    \]
    On the other hand, Gauss--Bonnet gives
    \[
        \int_{\partial\Omega} \Hn\dd S=-2\pi\chi(\Omega).
    \]
    Integrating the boundary condition over $\partial\Omega$ and using these two identities yields
    \[
        -2\pi
        =
        \int_{\partial\Omega} \partial_\nu u\dd S
        =
        \Gamma\int_{\partial\Omega} \Hn\dd S+(\Gamma-1)|\partial\Omega|
        =
        -2\pi\Gamma\chi(\Omega)+(\Gamma-1)|\partial\Omega|,
    \]
    which is \eqref{eq:planar-topological}.
\end{proof}

\bibliographystyle{plain}

\end{document}